\documentclass[12pt,reqno]{amsart}
\usepackage{graphicx}
\usepackage[unicode=true,bookmarks=false,breaklinks=false,pdfborder={0 0 1},colorlinks=false]{hyperref}
\hypersetup{unicode,colorlinks,urlcolor=blue,citecolor=blue}
 \usepackage[sort&compress]{natbib}\setcitestyle{numbers,square,comma}
\numberwithin{equation}{section}
\newtheorem{thm}{Theorem}[section]
\newtheorem{defn}[thm]{Definition}
\newtheorem{lem}[thm]{Lemma}
\newtheorem{cor}[thm]{Corollary}

\newtheorem{rem}[thm]{Remark}
\newtheorem{ex}[thm]{Example}
\DeclareMathOperator{\id}{id}

\begin{document}
\title[Set-theoretical solutions to Hom-Yang-Baxter equation]{Set-theoretical solutions to the Hom-Yang-Baxter equation and Hom-cycle sets}
\author{Kaiqiang Zhang}
\author{Xiankun Du}
\address{K. Zhang: School of Mathematics,  Jilin University, Changchun 130012, China}
\email{zkaiqiang@163.com}
\address{X. Du: School of Mathematics,  Jilin University, Changchun 130012,  China}
\email{duxk@jlu.edu.cn}

\thanks{Corresponding author: X Du}
\subjclass{16T25, 20N02, 20N05}
\keywords{cycle set, Hom-cycle set, Hom-Yang-Baxter equation, left quasigroup,  set-theoretic solution}
\begin{abstract}
Set-theoretic solutions to the Yang-Baxter equation have been studied extensively  by means of related algebraic systems such as cycle sets and braces, dynamical versions of which have also been developed. No work focuses on set-theoretic solutions to the Hom-Yang-Baxter equation (HYBE for short). This paper  investigates set-theoretic solutions to HYBE and associated algebraic system, called  Hom-cycle sets. We characterize  left non-degenerate involutive set-theoretic solutions to HYBE and Hom-cycle sets, and establish their relations. We discuss connections among Hom-cycle sets, cycle sets, left non-degenerate involutive set-theoretic solutions to HYBE and  the Yang-Baxter equation.
\end{abstract}

\maketitle

\section{Introduction}

Let $V$ be a vector space.  A solution to the Yang-Baxter equation (YBE shortly)  is a linear map $R:V\otimes V\rightarrow V\otimes V$
such that
\[
(R\otimes\id_{V})(\id_{V}\otimes R)(R\otimes\id_{V})=(\id_{V}\otimes R)(R\otimes\id_{V})(\id_{V}\otimes R).
\]
YBE first appeared in the work of Yang~\cite{Yang} and Baxter~\cite{Baxter}. This equation is fundamental to quantum groups. It is a central task to determine all solutions to YBE. However,
this is  difficult to accomplish. In order to find new solutions to YBE, Drinfeld~\cite{Dr} in 1992 suggested considering  set-theoretic solutions to YBE, that is, a map $r:X\times X\rightarrow X\times X$,
where $X$ is a nonempty set, satisfying
\begin{equation}
(r\times\id_{X})(\id_{X}\times r)(r\times\id_{X})=(\id_{X}\times r)(r\times\id_{X})(\id_{X}\times r).\label{eqsol}
\end{equation}
  Etingof, Schedler and Soloviev~\cite{Etingof},  Gateva-Ivanova and Van den Bergh\\~\cite{Gateva1}, and Lu, Yan and Zhu~\cite{lu2000} initially conducted
a systematic study on this subject. They studied set-theoretic
solutions with invertibility, nondegeneracy and involutivity by using
group theory. Gateva-Ivanova in~\cite{Gateva4} introduced a combinatorial approach to discuss set-theoretic solutions and she conjectured that every
square-free, non-degenerate involutive set-theoretic solution $(X,r)$
is decomposable whenever $X$ is finite. This has been proved by Rump
in~\cite{Rump1}.

Left cycle sets were introduced by Rump~\cite{Rump1} to study left
non-degenerate involutive set-theoretic solutions to YBE.
Rump showed that there is a bijective correspondence
between   left non-degenerate involutive set-theoretic solutions to YBE and left cycle sets, and non-degenerate solutions correspond to non-degenerate left cycle sets. He also prove
that all finite left cycle sets are non-degenerate. The theory of cycle sets has been proved to be very useful in understanding the structure
of solutions to YBE (see for example~\cite{Bonatto21,castelli2018,castelli2019,castelli2020A,castelli2021,castelli2020O,Dehornoy1,lebed2017}). This theory has been greatly developed,  and inspires theory of braces~\cite{Cedo1,gateva2018,Guarnieri,Rump2,smoktunowicz2018}.

Another version of YBE, dynamical quantum Yang-Baxter equation, has been  studied~\cite{Etingofdynamical,etingof2001}, which is closely related to dynamical quantum groups~\cite{etingof1998}.  Their set-theoretic solution, called DYB maps, were proposed by Shibukawa in~\cite{Shibukawa1} and received a lot of attention (see for example~\cite{kamiya2011,matsumoto2018,shibukawa2016, veselov2007}).  Dynamical braces and dynamical cycle sets were introduced and related to  right non-degenerate unitary DYB maps~\cite{Matsumoto13,Rump9}.

The Hom-Yang-Baxter equation (HYBE shortly) was introduced by Yau~\cite{Yau} motivated by Hom-Lie algebras, which is related to  Hom-quantum groups~\cite{Yau3}. Many researchers have devoted considerable attention to HYBE  (see for example~\cite{chen2015,Jiao2018,Panaite,wang2022,Yau1,Yau2}). However, to our knowledge, no work concentrates on set-theoretical solutions to  GYBE.

The aim of this paper is to investigate left non-degenerate involutive
set-theoretic solutions to HYBE, corresponding algebraic systems,
 called Hom-cycle sets, and their relationship.

The paper is organized as follows. In Section 2, we
review some basic definitions and results, and  provide a general categorical framework for the following discussion. In Section 3,
we   characterize  left non-degenerate involutive set-theoretic
solutions to HYBE. In Section 4,
we introduce the notion of a Hom-cycle set, and prove that there exists a one to one correspondence between left Hom-cycle sets and left non-degenerate
involutive set-theoretic solutions to HYBE.
 Section 5 is devoted to relationship among Hom-cycle sets, cycle sets,  left non-degenerate involutive solutions to HYBE and YBE.

\section{Preliminaries}
 Let $X$ be a nonempty set and let $r:X\times X\rightarrow X\times X$
be a map. We will write $r(x,y)=(\lambda_{x}(y),\rho_{y}(x))$, where
$\lambda_{x}$ and $\rho_{y}$ are maps from $X$ to itself for all
$x,y\in X$.  The pair $(X,r)$ is referred to a quadratic set in~\cite{Gateva4}.

A quadratic set $(X,r)$ (or a  map $r$) is called
\begin{enumerate}
\item left (respectively, right) non-degenerate if the map $\lambda_{x}$ (respectively, $\rho_{x}$)
is bijective for all $x\in X$;
\item non-degenerate if $r$ is both left and right non-degenerate;
\item involutive if $r^{2}=\id$, the identify map;
\item a set-theoretic solution to YBE if $r$ satisfies~\eqref{eqsol}.
\end{enumerate}

The following lemma comes from~\cite[Proposition 1.6]{Etingof} (see
also~\cite[Lemma 2.4]{Gateva2008matched}).
\begin{lem}\label{sec2:mark1}
 \begin{enumerate}
\item A quadratic set $(X,r)$   is involutive if and only if
\begin{gather}
\lambda_{\lambda_{x}(y)}\rho_{y}(x)=x,\label{sec2:eq1}\\
\rho_{\rho_{x}(y)}\lambda_{y}(x)=x,\label{sec2:eq2}
\end{gather}
for all $x,y\in X$.
\item\label{lndequiv}  A quadratic set $(X,r)$  is left non-degenerate and involutive if and only if $\lambda_{x}$ is bijective for all $x\in X$ and
\begin{equation}
\rho_{y}(x)=\lambda_{\lambda_{x}(y)}^{-1}(x),\label{eqInvolutive}
\end{equation}
for all $x,y\in X$.
\end{enumerate}
\end{lem}

\begin{thm}~\cite[Proposition 1.6]{Etingof}\label{YBEs}
A quadratic set $(X,r)$
is a set-theoretic solution to YBE if and only if
\begin{enumerate}
\item $\lambda_{\lambda_{x}(y)}\lambda_{\rho_{y}(x)}(z)=\lambda_{x}\lambda_{y}(z)$,
\item $\rho_{\lambda_{\rho_{y}(x)}(z)}\lambda_{x}(y)=\lambda_{\rho_{\lambda_{y}(z)}(x)}\rho_{z}(y)$,
\item $\rho_{z}\rho_{y}(x)=\rho_{\rho_{z}(y)}\rho_{\lambda_{y}(z)}(x)$,
\end{enumerate}
for all $x,y,z\in X$.
\end{thm}

The following Theorem comes from~\cite[Proposition 2]{Cedo1} (see also~\cite[Theorem 9.3.10]{Jesper}).
\begin{thm}\label{LNDISYBE}
A quadratic set  $(X,r)$ is
a left non-degenerate involutive set-theoretic solution to YBE
if and only if the following  hold.
\begin{enumerate}
\item $\lambda_{x}$ is bijective for all $x\in X$;
\item $r^{2}=\id$;
\item $\lambda_{x}\lambda_{\lambda_{x}^{-1}(y)} =\lambda_{y}\lambda_{\lambda_{y}^{-1}(x)}$
for all $x,y\in X$.
\end{enumerate}
\end{thm}

Let $(X,r)$ and $(X^{\prime},r^{\prime})$ be quadratic sets. By
a morphism from $(X,r)$ to $(X^{\prime},r^{\prime})$ we mean a map
$f:X\rightarrow X^{\prime}$ satisfying $(f\times f)r=r^{\prime}(f\times f)$.

\begin{lem}\label{lemmorph}
Given two quadratic sets $(X,r)$ and $(X^{\prime},r^{\prime})$,
and a map $f:X\rightarrow X^{\prime}$, the following are equivalent:
\begin{enumerate}
\item $f$ is a morphism of quadratic sets;
\item $f\lambda_{x}=\lambda_{f(x)}^{\prime}f~\text{and}~f\rho_{x}=\rho_{f(x)}^{\prime}f~\text{for all}~x\in X$.
\end{enumerate}
If $r$ and $r^{\prime}$ are both left non-degenerate and involutive,
then both conditions above are equivalent to one of the following
conditions:
\begin{enumerate}
\item[(3)] $f\lambda_{x}=\lambda_{f(x)}^{\prime}f$ for all $x\in X$;
\item[(4)] $f\lambda_{x}^{-1}=\lambda_{f(x)}^{\prime-1}f$ {for all} $x\in X$.
\end{enumerate}
\end{lem}

By a Hom-quadratic set we mean a triple $(X,r,\alpha)$ of a nonempty
set $X$ with two maps $r:X\times X\to X\times X$ and $\alpha:X\to X$
such that $r(\alpha\times\alpha)=(\alpha\times\alpha)r$. Thus a Hom-quadratic
set is exactly a quadratic set with an endomorphism.

We will identify the quadratic set $(X,r)$ with the Hom-quadratic set
$(X,r,\id)$.

Given two Hom-quadratic sets $(X,r,\alpha)$ and $(X^{\prime},r^{\prime},\alpha^{\prime})$,
a map $f:X\rightarrow X^{\prime}$ is called a morphism of Hom-quadratic
sets if
\[
(f\times f)r=r^{\prime}(f\times f)~\text{and}~f\alpha=\alpha^{\prime}f.
\]
Thus a morphism of Hom-quadratic sets is exactly a morphism of quadratic
sets satisfying $f\alpha=\alpha^{\prime}f$.
\begin{cor}\label{corHQS}
Given a quadratic set $(X,r)$ and a map $\alpha:X\to X$,
the triple $(X,r,\alpha)$ is a Hom-quadratic set if and only if
\[
\alpha\lambda_{x}=\lambda_{\alpha(x)}\alpha~\text{and}~\alpha\rho_{x}=\rho_{\alpha(x)}\alpha~\text{for all}~x\in X.
\]
\end{cor}

A Hom-quadratic set $(X,r,\alpha)$  is called left non-degenerate, non-degenerate, and involutive, respectively, if $r$ has the same properties.

Denote by $\mathsf{QS}$ and $\mathsf{HQS}$ the categories of left non-degenerate involutive quadratic sets and left non-degenerate involutive
  Hom-quadratic sets, respectively. Then $\mathsf{QS}$ is a full subcategory of  $\mathsf{HQS}$ by identifying a quadratic set $(X,r)$ with a Hom-quadratic set $(X,r,\id)$.

By a groupoid we mean a set with a binary operation. For a groupoid
$X$, denote by $\sigma_{x}$ the left multiplication map by $x\in X$
defined by
\[
\sigma_{x}:X\to X,~~y\mapsto xy.
\]

By a left quasigroup we mean a groupoid $X$ such that the left multiplication
maps $\sigma_{x}$ are bijective for all $x\in X$ (see~\cite[Page 9]{Shcherbacov}).

It should be pointed out that the image of an endomorphism of a left
quasigroup is a left quasigroup, though the image of a homomorphism
from a left quasigroup to a groupoid need not to be a left quasigroup (see~\cite[Page 15]{Shcherbacov} and~\cite[Corollary 1.298]{Shcherbacov}).

By a left Hom-quasigroup we mean a pair $(X,\alpha)$ of a left quasigroup
$X$ with an endomorphism $\alpha$.

We also write a left Hom-quasigroup $(X,\alpha)$ as $(X,\cdot,\alpha)$ to indicate the operation $\cdot$ of left quasigroup $X$.

We can identify a left quasigroup $X$ with the left Hom-quasigroup
$(X,\id)$.

Let $(X,\alpha)$ and $(X^{\prime},\alpha^{\prime})$ be two left
Hom-quasigroups. A map $f:X\to X^{\prime}$ is called a morphism of
left Hom-quasigroups if $\alpha f=f\alpha^{\prime}$ and $f(xy)=f(x)f(y)$
for all $x,y\in X$.

From  Lemma~\ref{lemmorph}, we have the following lemma.
\begin{lem}
Let $(X,r,\alpha)$ and $(X^{\prime},r^{\prime},\alpha^{\prime})$
be left non-degenerate involutive Hom-quadratic sets,  and let $(X,\cdot,\alpha)$ and $(X',*,\alpha')$ be left Hom-quasigroups such that $x\cdot y=\lambda_{x}^{-1}(y)$ for all $x,y\in X$ and $x'*y'=\lambda'^{-1}_{x'}(y')$ for all $x',y'\in X'$. Then a map $f: X \rightarrow X'$ is  a morphism of  Hom-quadratic sets if and only if it is a morphism of left Hom-quasigroups.
\end{lem}

Denote by $\mathsf{HQG}$ and $\mathsf{QG}$ the categories of left quasigroups and  left Hom-quasigroups, respectively. Then $\mathsf{HQG}$ is a full subcategory of $\mathsf{QG}$ by identifying a left   quasigroup $X$ with a left Hom-quasigroup $(X,\id)$.

Given a left non-degenerate involutive Hom-quadratic set $(X,r,\alpha)$,
we get a left Hom-quasigroups $(X,\cdot,\alpha)$, denoted by $G(X,r,\alpha)$,
with the operation defined by $x\cdot y=\lambda_{x}^{-1}(y)$ for
all $x,y\in X$. Then we have a functor $G:\mathsf{HQS}\to\mathsf{HQG}$
by associating $(X,r,\alpha)$ with $G(X,r,\alpha)$, and a morphism $f$ in $\mathsf{HQS}$
with $f$.

Conversely, given a left Hom-quasigroup $(X,\cdot,\alpha)$, we get
a Hom-quadratic set $(X,r,\alpha)$, denoted by $S(X,\cdot,\alpha)$,
with $\lambda_{x}(y)=\sigma_{x}^{-1}(y)$ and $\rho_{y}(x)=\sigma_{x}^{-1}(y)\cdot x$
for all $x,y\in X$. It is routine to verify that $(X,r,\alpha)$
is left non-degenerate and involutive. Then we have a functor $S:\mathsf{HQG}\to\mathsf{HQS}$
by associating  $(X,\cdot,\alpha)$
with $S(X,\cdot,\alpha)$ and a morphism $f$ in $\mathsf{HQG}$
with $f$.

\begin{thm}\label{LemlnditP}
The functors $G$ and $S$ are mutually inverse,
and so the categories $\mathsf{HQS}$ and $\mathsf{HQG}$ are
isomorphic.
\end{thm}
\begin{proof}
It is straightforward.
\end{proof}

 The functor $G$ induces a functor
from $\mathsf{QS}$ to $\mathsf{QG}$ and $S$ induces a
functor from $\mathsf{QG}$ to $\mathsf{QS}$. We still denote
the induced functors by $G$ and $S$, respectively. By Theorem~\ref{LemlnditP}, we have the following corollary.

\begin{cor}\label{LemlndiqLqg}
The functors $G: \mathsf{QS}\to\mathsf{QG}$ and $S: \mathsf{QG}\to\mathsf{QS}$ are mutually inverse,
and so the categories $\mathsf{QS}$ and $\mathsf{QG}$ are isomorphic.
\end{cor}

Denote by $\mathsf{S_{ybe}}$ the category of left non-degenerate
involutive solutions to YBE. Then $\mathsf{S_{ybe}}$ is a full
subcategory of $\mathsf{QS}$.

A left quasigroup $X$ is called a left cycle set if $(xy)(xz)=(yx)(yz)$
for all $x,y,z\in X$~\cite{Rump1}.

Denote by $\mathsf{CS}$ the category of left cycle sets. Then $\mathsf{CS}$
is a full subcategory of $\mathsf{QG}$.

By~\cite[Proposition 1]{Rump1}, the functor $G$ induces a functor
from $\mathsf{S_{ybe}}$ to $\mathsf{CS}$ and $S$ induces a
functor from $\mathsf{CS}$ to $\mathsf{S_{ybe}}$. We still denote
the induced functors by $G$ and $S$, respectively. By Corollary
~\ref{LemlndiqLqg},~\cite[Proposition 1]{Rump1} can be restated as
follows.

\begin{thm}\label{TemLndisC}
The functors $G:\mathsf{S_{ybe}}\to\mathsf{CS}$
and $S:\mathsf{CS}\to\mathsf{S_{ybe}}$ are mutually inverse,
and so the categories $\mathsf{S_{ybe}}$ and $\mathsf{CS}$ are
isomorphic.
\end{thm}

A groupoid is called $\Delta$-bijective if  the map $\Delta$ is bijective, where
$\Delta:X\times X\to X\times X,~~(x,y)\mapsto(xy,yx)$~\cite{Bonatto21}.

The following lemma comes from~\cite[Lemma 2.10]{Bonatto21} (see
also Lemma 1.28 in Chapter XIII of~\cite{Dehornoy3}).

\begin{lem}\label{lemPhi}
A groupoid $X$ is $\Delta$-bijective if and only if
there exists an operation $\circ$ on $X$ (called the dual operation) such that
\begin{gather}
(x\cdot y)\circ(y\cdot x)=x,\label{eqxyyxd}\\
(x\circ y)\cdot(y\circ x)=x,\label{eqxyyxc}
\end{gather}
for all $x,y\in X$.
Furthermore, if  conditions hold, then the operation $\circ$
is unique and the inverse of $\Delta$ is given by $(x,y)\mapsto(x\circ y,y\circ x)$.
\end{lem}

\begin{lem}~(see~\cite[Lemma 2.11]{Bonatto21})\label{lem2div}
If a groupoid
$X$ is $\Delta$-bijective, then the square map $q:X\to X,~~x\mapsto x^{2}$
is invertible.
\end{lem}

We will say that a groupoid with extra structure is non-degenerate
if the underlying groupoid is $\Delta$-bijective.

Denote by $\mathsf{ndCS}$ the category of non-degenerate left cycle
sets. Then $\mathsf{ndCS}$ is   a full subcategory of $\mathsf{QG}$.

Denote by $\mathsf{ndiS_{ybe}}$ the category of non-degenerate involutive
solution to YBE. Then $\mathsf{ndiS_{ybe}}$ is a full subcategory
of $\mathsf{QS}$.

By Corollary~\ref{LemlndiqLqg} and Theorem~\ref{TemLndisC},  Proposition
2 in~\cite{Rump1} can be restated as follows.

\begin{thm}
The functors $G:\mathsf{ndiS_{ybe}}\to\mathsf{ndCS}$
and $S:\mathsf{ndCS}\to\mathsf{ndiS_{ybe}}$ are mutually inverse,
and so the categories $\mathsf{ndiS_{ybe}}$ and $\mathsf{ndCS}$
are isomorphic.
\end{thm}

\section{Set-theoretic solutions to the Hom-Yang-Baxter equation}

The Hom-Yang-Baxter equation was   proposed by Yau~\cite{Yau} motivated
by Hom-Lie algebras.

\begin{defn}
Given a vector space $V$ and two linear maps $R:V\otimes V\rightarrow V\otimes V$
and $\alpha:V\rightarrow V$, the triple $(V,R,\alpha)$ is called a solution to the Hom-Yang-Baxter equation, if
\begin{enumerate}
\item $R(\alpha\otimes\alpha)=(\alpha\otimes\alpha)R$, and
\item $(\alpha\otimes R)(R\otimes\alpha)(\alpha\otimes R)=(R\otimes\alpha)(\alpha\otimes R)(R\otimes\alpha)$.
\end{enumerate}
\end{defn}

By analogy with set-theoretic solutions to YBE, we introduce set-theoretic
solutions to HYBE.

\begin{defn}
Given a nonempty set $X$ and two maps $r:X\times X\rightarrow X\times X$
and $\alpha:X\rightarrow X$, the triple $(X,r,\alpha)$ is called a set-theoretic
solution to HYBE, if
\begin{enumerate}
\item $r\circ(\alpha\times\alpha)=(\alpha\times\alpha)\circ r$, and
\item $(\alpha\times r)(r\times\alpha)(\alpha\times r)=(r\times\alpha)(\alpha\times r)(r\times\alpha)$.
\end{enumerate}
\end{defn}

Clearly, $(X,r)$ is a set-theoretic solution to YBE if and only if
$(X,r,\id)$ is a set-theoretic solution to HYBE.

Let $f:X\to X$ be a map and $Y$ a subset of $X$. Denote by $f|_{Y}$
the restriction of $f$ to $Y$. When there is no ambiguity, we will write $f$ for the restriction $f|_{Y}$.

By analogy with the relation between set-theoretic solutions to YBE
and solutions to YBE, we have the following theorem, and the proof   is immediate.

\begin{thm}
Let $V$ be a vector space with a basis $X$.
\begin{enumerate}
\item If $(V,R,\alpha)$ is a solution to HYBE such that $R(X\otimes X)\subseteq X\otimes X$
and $\alpha(X)\subseteq X$, then $(X,R|_{X\otimes X},\alpha|_{X})$
is a set-theoretic solution to  HYBE.
\item Conversely, if $(X,r,\alpha)$ is a set-theoretic solution to HYBE,
and $R:V\rightarrow V$ and $\bar{\alpha}:V\to V$ are linear extensions
of $r$ and $\alpha$, respectively, then $(V,R,\bar{\alpha})$ is
a solution to HYBE.
\end{enumerate}
\end{thm}

In what follows, a set-theoretic solution is simply called a solution.

\begin{lem}\label{Sec3LemSHYBE}
A triple $(X,r,\alpha)$  with $r:X\times X\rightarrow X\times X$ and $\alpha:X\rightarrow X$ is a solution to HYBE
if and only if the following conditions hold for all $x,y,z\in X$,
\begin{enumerate}
\item\label{Sec3eqHybe1} $\alpha\lambda_{x}=\lambda_{\alpha(x)}\alpha,\text{and}~\alpha\rho_{x}=\rho_{\alpha(x)}\alpha$;
\item\label{Sec3eqHybe2} $\alpha\lambda_{\alpha(x)}\lambda_{y}=\lambda_{\alpha\lambda_{x}(y)}\lambda_{\rho_{y}(x)}\alpha$;
\item\label{Sec3eqHybe3} $\rho_{\lambda_{\rho_{y}(x)}\alpha(z)}\alpha\lambda_{x}(y)=\lambda_{\rho_{\lambda_{y}(z)}\alpha(x)}\alpha\rho_{z}(y)$;
\item\label{Sec3eqHybe4} $\alpha\rho_{\alpha(y)}\rho_{x}=\rho_{\alpha\rho_{y}(x)}\rho_{\lambda_{x}(y)}\alpha$.
\end{enumerate}
\end{lem}

\begin{proof}
It is straightforward.
\end{proof}
Theorem~\ref{YBEs} is a special case of Lemma~\ref{Sec3LemSHYBE}.

\begin{ex}
A triple $(X,\id_{X\times X},\alpha)$ is a solution to
HYBE if and only if $\alpha^{2}=\alpha$. In this case, it is involutive,
but neither left nor right non-degenerate.
\end{ex}

\begin{ex}
Let $(X,r,\alpha)$ be Hom-quadratic
set. If $(X,r)$ is a solution to YBE, then $(X,(\alpha\times\alpha)r,\alpha)$
is a solution to HYBE, and the converse holds if additionally $\alpha$
is injective or surjective.
\end{ex}

\begin{ex}
A triple $(X,\tau,\alpha)$ with $\tau(x,y)=(y,x)$ and
arbitrary map $\alpha:X\to X$ is a non-degenerate involutive solution to HYBE, called a trivial solution.
\end{ex}

\begin{ex}\label{exPermutation}
A triple $(X,r,\alpha)$ with $r(x,y)=(f(y),g(x))$, where
$f,g$ are maps from $X$ to itself, is a solution to HYBE if and only if $\alpha ,f,g$ commute.  Furthermore, the solution $(X,r,\alpha)$ is   left non-degenerate and involutive  if and only if $f$ is bijective, and $g=f^{-1}$.
 \end{ex}
We are now in a position to  characterize left non-degenerate
involutive solutions to HYBE.
\begin{thm}\label{Sec3ThmLNDIeqV}
A triple $(X,r,\alpha)$ with $r:X\times X\rightarrow X\times X$ and $\alpha:X\rightarrow X$  is a left non-degenerate
involutive solution to HYBE if and only if the following conditions
hold for all $x,y\in X$,
\begin{enumerate}
\item\label{eqlnishybe1} $\lambda_{x}$ is bijective;
\item\label{eqlnishybe3} $\rho_{y}(x)=\lambda_{\lambda_{x}(y)}^{-1}(x)$;
\item\label{eqlnishybe2} $\alpha\lambda_{x}=\lambda_{\alpha(x)}\alpha$;
\item\label{eqlnishybe4} $\alpha\lambda_{\alpha(x)}\lambda_{\lambda_{x}^{-1}(y)}=\lambda_{\alpha(y)}\lambda_{\lambda_{y}^{-1}(x)}\alpha$;
\item\label{eqlnishybe5} $\alpha\lambda_{x}\lambda_{\lambda_{\alpha(x)}^{-1}(y)}=\lambda_{\alpha(y)}\lambda_{\lambda_{y}^{-1}\alpha(x)}\alpha$;
\item\label{eqlnishybe6} $\alpha\lambda_{x}\lambda_{\lambda_{\alpha(x)}^{-1}\alpha(y)}=\lambda_{y}\lambda_{\lambda_{\alpha(y)}^{-1}\alpha(x)}\alpha$.
\end{enumerate}
\end{thm}

\begin{proof}
($\Rightarrow$)~\eqref{eqlnishybe1} and~\eqref{eqlnishybe3} follows
from Lemma~\ref{sec2:mark1}\eqref{lndequiv}.~\eqref{eqlnishybe2}
follows from Lemma~\ref{Sec3LemSHYBE}\eqref{Sec3eqHybe1}.

To prove~\eqref{eqlnishybe4}, replacing $y$ by $\lambda_{x}^{-1}(y)$
in~\eqref{eqlnishybe3} and Lemma~\ref{Sec3LemSHYBE}\eqref{Sec3eqHybe2},
we get
\[
\rho_{\lambda_{x}^{-1}(y)}(x)=\lambda_{y}^{-1}(x)~~\text{and}~~\alpha\lambda_{\alpha(x)}\lambda_{\lambda_{x}^{-1}(y)}=\lambda_{\alpha(y)}\lambda_{\rho_{\lambda_{x}^{-1}(y)}(x)}\alpha.
\]
Then~\eqref{eqlnishybe4} follows.

To prove~\eqref{eqlnishybe5}, using~\eqref{eqlnishybe3} we can write
Lemma~\ref{Sec3LemSHYBE}\eqref{Sec3eqHybe3} as
\[
\lambda_{\lambda_{\alpha\lambda_{x}(y)}\lambda_{\rho_{y}(x)}\alpha(z)}^{-1}\alpha\lambda_{x}(y)=\lambda_{\rho_{\lambda_{y}(z)}\alpha(x)}\alpha\lambda_{\lambda_{y}(z)}^{-1}(y).
\]
It follows by Lemma~\ref{Sec3LemSHYBE}\eqref{Sec3eqHybe2} that
\begin{equation*}
\lambda_{\alpha\lambda_{\alpha(x)}\lambda_{y}(z)}^{-1}\alpha\lambda_{x}(y)=\lambda_{\rho_{\lambda_{y}(z)}\alpha(x)}\alpha\lambda_{\lambda_{y}(z)}^{-1}(y),
\end{equation*}
which implies that $\alpha\lambda_{x}(y)=\lambda_{\alpha\lambda_{\alpha(x)}\lambda_{y}(z)}\lambda_{\rho_{\lambda_{y}(z)}\alpha(x)}\alpha\lambda_{\lambda_{y}(z)}^{-1}(y)$.
Since $\lambda_{y}$ is bijective and $z$ is arbitrary, we can replace
$\lambda_{y}(z)$ by $z$ in the last equation to obtain $\alpha\lambda_{x}(y)=\lambda_{\alpha\lambda_{\alpha(x)}(z)}\lambda_{\rho_{z}\alpha(x)}\alpha\lambda_{z}^{-1}(y)$.
Thus \[\alpha\lambda_{x}=\lambda_{\alpha\lambda_{\alpha(x)}(z)}\lambda_{\rho_{z}\alpha(x)}\alpha\lambda_{z}^{-1},\]
and so $\alpha\lambda_{x}\lambda_{z}=\lambda_{\alpha\lambda_{\alpha(x)}(z)}\lambda_{\rho_{z}\alpha(x)}\alpha$.
By~\eqref{eqlnishybe3} we have \[\alpha\lambda_{x}\lambda_{z}=\lambda_{\alpha\lambda_{\alpha(x)}(z)}\lambda_{\lambda_{\lambda_{\alpha(x)}(z)}^{-1}\alpha(x)}\alpha,\]
in which replacing $z$ by $\lambda_{\alpha(x)}^{-1}(y)$ gives~\eqref{eqlnishybe5}.

Now we prove~\eqref{eqlnishybe6}. We first claim that the conditions~\eqref{eqlnishybe1} through~\eqref{eqlnishybe5} imply that
\begin{equation}
\alpha\lambda_{x}=\alpha\lambda_{\alpha^{2}(x)}.\label{ThmEqdeRive1}
\end{equation}
Indeed, replacing $x$ by $\alpha(x)$ in~\eqref{eqlnishybe4}, we
have
\[
\alpha\lambda_{\alpha^{2}(x)}\lambda_{\lambda_{\alpha(x)}^{-1}(y)}=\lambda_{\alpha(y)}\lambda_{\lambda_{y}^{-1}\alpha(x)}\alpha,
\]
which together with~\eqref{eqlnishybe5} gives $\alpha\lambda_{x}\lambda_{\lambda_{\alpha(x)}^{-1}(y)}=\alpha\lambda_{\alpha^{2}(x)}\lambda_{\lambda_{\alpha(x)}^{-1}(y)}$.
By~\eqref{eqlnishybe1} we get $\alpha\lambda_{x}=\alpha\lambda_{\alpha^{2}(x)}$.

Note that Lemma~\ref{Sec3LemSHYBE}\eqref{Sec3eqHybe2} implies that
\begin{equation}
\lambda_{\rho_{x}(z)}\alpha=\lambda_{\alpha\lambda_{z}(x)}^{-1}\alpha\lambda_{\alpha(z)}\lambda_{x}.\label{ThmEqTudao2}
\end{equation}
By~\eqref{eqlnishybe3}, we can write~\ref{Sec3LemSHYBE}\eqref{Sec3eqHybe4}
as
\begin{align*}
\alpha\lambda_{\lambda_{\rho_{x}(z)}\alpha(y)}^{-1}\rho_{x}(z)=\lambda_{\lambda_{\rho_{\lambda_{x}(y)}\alpha(z)}\alpha\rho_{y}(x)}^{-1}\rho_{\lambda_{x}(y)}\alpha(z),
\end{align*}
for any $z\in X$. It follows by~\eqref{eqlnishybe3} that
\begin{align*}
\alpha\lambda_{\lambda_{\rho_{x}(z)}\alpha(y)}^{-1}\lambda_{\lambda_{z}(x)}^{-1}(z)=\lambda_{\lambda_{\rho_{\lambda_{x}(y)}\alpha(z)}\alpha\lambda_{\lambda_{x}(y)}^{-1}(x)}^{-1}\lambda_{\lambda_{\alpha(z)}\lambda_{x}(y)}^{-1}\alpha(z).
\end{align*}
By using~\eqref{ThmEqTudao2} and substituting $\lambda_{\rho_{x}(z)}\alpha$
and $\lambda_{\rho_{\lambda_{x}(y)}\alpha(z)}\alpha$ into the last equation,
we obtain
\[
\alpha\lambda_{\lambda_{\alpha\lambda_{z}(x)}^{-1}\alpha\lambda_{\alpha(z)}\lambda_{x}(y)}^{-1}\lambda_{\lambda_{z}(x)}^{-1}(z)=\lambda_{\lambda_{\alpha\lambda_{\alpha(z)}\lambda_{x}(y)}^{-1}\alpha\lambda_{\alpha^{2}(z)}(x)}^{-1}\lambda_{\lambda_{\alpha(z)}\lambda_{x}(y)}^{-1}\alpha(z).
\]
Thus by~\eqref{ThmEqdeRive1}, we have
\[
\alpha\lambda_{\lambda_{\alpha\lambda_{z}(x)}^{-1}\alpha\lambda_{\alpha(z)}\lambda_{x}(y)}^{-1}\lambda_{\lambda_{z}(x)}^{-1}(z)=\lambda_{\lambda_{\alpha\lambda_{\alpha(z)}\lambda_{x}(y)}^{-1}\alpha\lambda_{z}(x)}^{-1}\lambda_{\lambda_{\alpha(z)}\lambda_{x}(y)}^{-1}\alpha(z).
\]
Replacing $x$ by $\lambda_{z}^{-1}(x)$ in the previous equation,
we have
\[
\alpha\lambda_{\lambda_{\alpha(x)}^{-1}\alpha\lambda_{\alpha(z)}\lambda_{\lambda_{z}^{-1}(x)}(y)}^{-1}\lambda_{x}^{-1}(z)=\lambda_{\lambda_{\alpha\lambda_{\alpha(z)}\lambda_{\lambda_{z}^{-1}(x)}(y)}^{-1}\alpha(x)}^{-1}\lambda_{\lambda_{\alpha(z)}\lambda_{\lambda_{z}^{-1}(x)}(y)}^{-1}\alpha(z).
\]
Noting that $\lambda_{\alpha(z)}\lambda_{\lambda_{z}^{-1}(x)}^{-1}(y)$
is an arbitrary element of $X$, we may simply denote it by $y$.
Then the last equation can be written as
\[
\alpha\lambda_{\lambda_{\alpha(x)}^{-1}\alpha(y)}^{-1}\lambda_{x}^{-1}(z)=\lambda_{\lambda_{\alpha(y)}^{-1}\alpha(x)}^{-1}\lambda_{y}^{-1}\alpha(z),
\]
that is, $\alpha\lambda_{\lambda_{\alpha(x)}^{-1}\alpha(y)}^{-1}\lambda_{x}^{-1}=\lambda_{\lambda_{\alpha(y)}^{-1}\alpha(x)}^{-1}\lambda_{y}^{-1}\alpha$,
which implies
\[
\alpha\lambda_{x}\lambda_{\lambda_{\alpha(x)}^{-1}\alpha(y)}=\lambda_{y}\lambda_{\lambda_{\alpha(y)}^{-1}\alpha(x)}\alpha.
\]
This proves~\eqref{eqlnishybe6}.

($\Leftarrow$) The nondegeneracy and involutivity of $(X,r,\alpha)$
follows from~\eqref{eqlnishybe1} and~\eqref{eqlnishybe3} by Lemma~\ref{sec2:mark1}\eqref{lndequiv}. We now prove that $(X,r,\alpha)$ satisfies the four conditions in Lemma~\ref{Sec3LemSHYBE}.

Lemma~\ref{Sec3LemSHYBE}\eqref{Sec3eqHybe1} follows from Lemma~\ref{lemmorph}.

By~\eqref{eqlnishybe3} and~\eqref{eqlnishybe4}, we have
\[
\lambda_{\alpha\lambda_{x}(y)}\lambda_{\rho_{y}(x)}\alpha=\lambda_{\alpha\lambda_{x}(y)}\lambda_{\lambda_{\lambda_{x}(y)}^{-1}(x)}\alpha=\alpha\lambda_{\alpha(x)}\lambda_{\lambda_{x}^{-1}\lambda_{x}(y)}=\alpha\lambda_{\alpha(x)}\lambda_{y},
\]
which proves Lemma~\ref{Sec3LemSHYBE}\eqref{Sec3eqHybe2}.

To prove Lemma~\ref{Sec3LemSHYBE}\eqref{Sec3eqHybe3}, replacing
$x$ by $\alpha(x)$ and $y$ by $\lambda_{y}(z)$ in Lemma~\ref{Sec3LemSHYBE}\eqref{Sec3eqHybe2} and using~\eqref{ThmEqdeRive1} we get
\[
\lambda_{\alpha\lambda_{\alpha(x)}\lambda_{y}(z)}\lambda_{\rho_{\lambda_{y}(z)}\alpha(x)}\alpha=\alpha\lambda_{\alpha^{2}(x)}\lambda_{\lambda_{y}(z)}=\alpha\lambda_{x}\lambda_{\lambda_{y}(z)}.
\]
It follows that
\begin{equation}\label{ThmEqdeRive5-2}
\lambda_{\alpha\lambda_{\alpha(x)}\lambda_{y}(z)}^{-1}\alpha\lambda_{x}=\lambda_{\rho_{\lambda_{y}(z)}\alpha(x)}\alpha\lambda_{\lambda_{y}(z)}^{-1}.
\end{equation}
Thus  by Lemma~\ref{Sec3LemSHYBE}\eqref{Sec3eqHybe2},~\eqref{ThmEqdeRive5-2} and ~\eqref{eqlnishybe3}, we have
\begin{align*}
\rho_{\lambda_{\rho_{y}(x)}\alpha(z)}\alpha\lambda_{x}(y) & =\lambda_{\lambda_{\alpha\lambda_{x}(y)}\lambda_{\rho_{y}(x)}\alpha(z)}^{-1}\alpha\lambda_{x}(y)\\
 & =\lambda_{\alpha\lambda_{\alpha(x)}\lambda_{y}(z)}^{-1}\alpha\lambda_{x}(y)\\
 & =\lambda_{\rho_{\lambda_{y}(z)}\alpha(x)}\alpha\lambda_{\lambda_{y}(z)}^{-1}(y)\\
 & =\lambda_{\rho_{\lambda_{y}(z)}(\alpha(x))}\alpha\rho_{z}(y),
\end{align*}
which proves Lemma~\ref{Sec3LemSHYBE}\eqref{Sec3eqHybe3}.

We now prove Lemma~\ref{Sec3LemSHYBE}\eqref{Sec3eqHybe4}. By Lemma~\ref{Sec3LemSHYBE}\eqref{Sec3eqHybe2},
we have
\begin{align}\label{ThmEqdeRive5-1}
\lambda_{\rho_{y}(x)}\alpha & =\lambda_{\alpha\lambda_{x}(y)}^{-1}\alpha\lambda_{\alpha(x)}\lambda_{y}.
\end{align}
And by~\eqref{ThmEqdeRive5-2}, we have
\begin{align}\label{ThmEqdeRive5-3}
\lambda_{\rho_{\lambda_{y}(z)}\alpha(x)}\alpha & =\lambda_{\alpha\lambda_{\alpha(x)}\lambda_{y}(z)}^{-1}\alpha\lambda_{x}\lambda_{\lambda_{y}(z)}.
\end{align}
Then using~\eqref{eqlnishybe3} and~\eqref{ThmEqdeRive5-1} we obtain
that
\begin{align}
\alpha\rho_{\alpha(z)}\rho_{y}(x) & =\alpha\lambda_{\lambda_{\rho_{y}(x)}\alpha(z)}^{-1}\rho_{y}(x)\nonumber \\
 & =\alpha\lambda_{\lambda_{\alpha\lambda_{x}(y)}^{-1}\alpha\lambda_{\alpha(x)}\lambda_{y}(z)}^{-1}\lambda_{\lambda_{x}(y)}^{-1}(x).\label{eq}
\end{align}
By~\eqref{eqlnishybe3} and~\eqref{ThmEqdeRive5-3} we get
\begin{align}
\rho_{\alpha\rho_{z}(y)}\rho_{\lambda_{y}(z)}\alpha(x) & =\lambda_{\lambda_{\rho_{\lambda_{y}(z)}\alpha(x)}\alpha\lambda_{\lambda_{y}(z)}^{-1}(y)}^{-1}\rho_{\lambda_{y}(z)}\alpha(x)\nonumber \\
 & =\lambda_{\lambda_{\alpha\lambda_{\alpha(x)}\lambda_{y}(z)}^{-1}\alpha\lambda_{x}\lambda_{\lambda_{y}(z)}\lambda_{\lambda_{y}(z)}^{-1}(y)}^{-1}\rho_{\lambda_{y}(z)}\alpha(x)\nonumber\\
 & =\lambda_{\lambda_{\alpha\lambda_{\alpha(x)}\lambda_{y}(z)}^{-1}\alpha\lambda_{x}(y)}^{-1}\lambda_{\lambda_{\alpha(x)}\lambda_{y}(z)}^{-1}\alpha(x).\label{eq2}
\end{align}
Substituting $y$ for $\alpha(y)$ in~\eqref{eqlnishybe5}, we have
$\alpha\lambda_{x}\lambda_{\lambda_{\alpha(x)}^{-1}\alpha(y)}=\lambda_{\alpha^{2}(y)}\lambda_{\lambda_{\alpha(y)}^{-1}\alpha(x)}\alpha$,
which together with~\eqref{eqlnishybe6} yields
\begin{equation}
\lambda_{\alpha^{2}(y)}\lambda_{\lambda_{\alpha(y)}^{-1}\alpha(x)}\alpha=\lambda_{y}\lambda_{\lambda_{\alpha(y)}^{-1}\alpha(x)}\alpha.\label{ThmEqdeRive3}
\end{equation}
Replacing $x$ by $\alpha\lambda_{x}(y)$ and $y$ by $\alpha\lambda_{\alpha(x)}\lambda_{y}(z)$ in~\eqref{eqlnishybe4}, respectively,  we have
\[
\alpha\lambda_{\alpha^{2}\lambda_{x}(y)}\lambda_{\lambda_{\alpha\lambda_{x}(y)}^{-1}\alpha\lambda_{\alpha(x)}\lambda_{y}(z)}=\lambda_{\alpha^{2}\lambda_{\alpha(x)}\lambda_{y}(z)}\lambda_{\lambda_{\alpha\lambda_{\alpha(x)}\lambda_{y}(z)}^{-1}\alpha\lambda_{x}(y)}\alpha.
\]
By~\eqref{ThmEqdeRive1} and~\eqref{ThmEqdeRive3}, we get that
\[
\alpha\lambda_{\lambda_{x}(y)}\lambda_{\lambda_{\alpha\lambda_{x}(y)}^{-1}\alpha\lambda_{\alpha(x)}\lambda_{y}(z)}=\lambda_{\lambda_{\alpha(x)}\lambda_{y}(z)}\lambda_{\lambda_{\alpha\lambda_{\alpha(x)}\lambda_{y}(z)}^{-1}\alpha\lambda_{x}(y)}\alpha,
\]
whence
\begin{equation}
\alpha\lambda_{\lambda_{\alpha\lambda_{x}(y)}^{-1}\alpha\lambda_{\alpha(x)}\lambda_{y}(z)}^{-1}\lambda_{\lambda_{x}(y)}^{-1}=\lambda_{\lambda_{\alpha\lambda_{\alpha(x)}\lambda_{y}(z)}^{-1}\alpha\lambda_{x}(y)}^{-1}\lambda_{\lambda_{\alpha(x)}\lambda_{y}(z)}^{-1}\alpha.\label{ThmEqdeRive4}
\end{equation}
 Lemma~\ref{Sec3LemSHYBE}\eqref{Sec3eqHybe4} follows from~\eqref{eq},~\eqref{eq2} and~\eqref{ThmEqdeRive4}.
\end{proof}

\begin{cor}\label{corxa2}
Let $(X,r,\alpha)$ be a  left non-degenerate involutive solution to HYBE. Then
\begin{enumerate}
\item\label{eqcor1} $\lambda_{x}\alpha=\lambda_{\alpha^{2}(x)}\alpha$ for all $x\in X$;
\item  $\lambda_{x}\alpha^2=\alpha^2\lambda_{x}$  for all $x\in X$.
 \end{enumerate}
 \end{cor}

\begin{proof}
(1)  Replacing $y$ by $\alpha(y)$ in Theorem~\ref{Sec3ThmLNDIeqV}\eqref{eqlnishybe5}
 we have \[\alpha\lambda_{x}\lambda_{\lambda_{\alpha(x)}^{-1}\alpha(y)} =\lambda_{\alpha^{2}(y)}\lambda_{\lambda_{\alpha(y)}^{-1}\alpha(x)}\alpha,\]
which together with Theorem~\ref{Sec3ThmLNDIeqV}\eqref{eqlnishybe6}
yields
\[
\lambda_{\alpha^{2}(y)}\lambda_{\lambda_{\alpha(y)}^{-1}\alpha(x)}\alpha=\lambda_{y}\lambda_{\lambda_{\alpha(y)}^{-1}\alpha(x)}\alpha.
\]
By Theorem~\ref{Sec3ThmLNDIeqV}\eqref{eqlnishybe2},  $\lambda_{\alpha^{2}(y)}\alpha\lambda_{\lambda_{y}^{-1}(x)} =\lambda_{y}\alpha\lambda_{\lambda_{y}^{-1}(x)}$.
Thus $\lambda_{\alpha^{2}(y)}\alpha\ =\lambda_{y}\alpha$.

(2)  By Theorem~\ref{Sec3ThmLNDIeqV}\eqref{eqlnishybe2} and Corollary~\ref{corxa2}\eqref{eqcor1}
we have \[\alpha\lambda_{\alpha(x)}=\lambda_{\alpha^{2}(x)}\alpha=\lambda_{x}\alpha.\]
By Theorem~\ref{Sec3ThmLNDIeqV}\eqref{eqlnishybe2},  we have $\alpha^{2}\lambda_{x}=\alpha\lambda_{\alpha(x)}\alpha=\lambda_{x}\alpha^{2}$.
\end{proof}

\begin{thm}\label{Sec3Thm2}
A triple $(X,r,\alpha)$ with $r:X\times X\rightarrow X\times X$ and $\alpha:X\rightarrow X$ is a left non-degenerate
involutive solution to HYBE if and only if the following statements
are true for all $x,y\in X$,
\begin{enumerate}
\item $\lambda_{x}\ $is bijective;
\item $\rho_{y}(x)=\lambda_{\lambda_{x}(y)}^{-1}(x)$;
\item $\alpha\lambda_{x}=\lambda_{\alpha(x)}\alpha$;
\item $\lambda_{x}\alpha=\alpha\lambda_{\alpha(x)}$;
\item $\alpha\lambda_{x}\lambda_{\lambda_{\alpha(x)}^{-1}(y)}=\lambda_{\alpha(y)}\lambda_{\lambda_{y}^{-1}\alpha(x)}\alpha$.
\end{enumerate}
\end{thm}

\begin{proof}
($\Rightarrow$)  By Theorem~\ref{Sec3ThmLNDIeqV} we need only prove (4). By (3) and Corollary \ref{corxa2}, $\alpha\lambda_{\alpha(x)}=\lambda_{\alpha^2(x)}\alpha=\lambda_x\alpha$, as desired.

($\Leftarrow$) It suffices to prove~\eqref{eqlnishybe4} and~\eqref{eqlnishybe6}
of Theorem~\ref{Sec3ThmLNDIeqV}. Replacing $y$ by $\alpha(y)$
in (5) and using (3) and (4) we get
\[
\alpha\lambda_{x}\lambda_{\lambda_{\alpha(x)}^{-1}\alpha(y)}=\lambda_{\alpha^{2}(y)}\lambda_{\lambda_{\alpha(y)}^{-1}\alpha(x)}\alpha=\lambda_{y}\lambda_{\lambda_{\alpha(y)}^{-1}\alpha(x)}\alpha,
\]
which is Theorem~\ref{Sec3ThmLNDIeqV}\eqref{eqlnishybe6}.

Replacing $x$ by $\alpha(x)$ in (5) we get $\alpha\lambda_{\alpha(x)}\lambda_{\lambda_{\alpha^{2}(x)}^{-1}(y)}=\lambda_{\alpha(y)}\lambda_{\lambda_{y}^{-1}\alpha^{2}(x)}\alpha$.
It follows from (3) and (4) that
\begin{multline*}
\lambda_{\alpha(y)}\lambda_{\lambda_{y}^{-1}\alpha^{2}(x)}\alpha=\alpha\lambda_{\alpha(x)}\lambda_{\lambda_{\alpha^{2}(x)}^{-1}(y)}=\lambda_{\alpha^{2}(x)}\lambda_{\lambda_{\alpha^{3}(x)}^{-1}\alpha(y)}\alpha\\
=\lambda_{\alpha^{2}(x)}\lambda_{\lambda_{\alpha(x)}^{-1}\alpha(y)}\alpha=\alpha\lambda_{\alpha(x)}\lambda_{\lambda_{x}^{-1}(y)}.
\end{multline*}
This proves Theorem~\ref{Sec3ThmLNDIeqV}\eqref{eqlnishybe4}.
\end{proof}
When $\alpha=\id$,  Theorem~\ref{Sec3Thm2} gives Theorem~\ref{LNDISYBE}.

\begin{ex}\label{exTheta}
A triple $(X,r,\alpha)$ such that $\alpha(X)=\{\theta\}$
is a left non-degenerate involutive solution to HYBE if and only if
$\lambda_{x}$ is bijective, $\lambda_{x}(\theta)=\theta$ and $\rho_{y}(x)=\lambda_{\lambda_{x}(y)}^{-1}(x)$
for all $x,y\in X$.
\end{ex}

\section{Hom-cycle sets}
In this section, we introduce a Hom-version of cycle
sets and relate them to  left non-degenerate involutive
solutions to HYBE.
\begin{defn}
A left Hom-quasigroup $(X,\alpha)$ is called a left Hom-cycle set if the
following conditions are satisfied for all $x,y,z\in X$,
\begin{gather}
\alpha((xy)(\alpha(x)z))=(yx)(\alpha(y)\alpha(z)),\label{Sec4ProEq1}\\
\alpha((\alpha(x)y)(xz))=(y\alpha(x))(\alpha(y)\alpha(z)),\label{Sec4ProEq2}\\
\alpha((\alpha(x)\alpha(y))(xz))=(\alpha(y)\alpha(x))(y\alpha(z)).\label{Sec4ProEq3}
\end{gather}
\end{defn}

In what follows, left cycle sets and left
Hom-cycle sets are referred to  cycle sets and Hom-cycle sets.

Clearly, $X$ is a cycle set if and only if $(X,\id_{X})$ is a Hom-cycle
set. Hence we will identify the  cycle set $X$ with the Hom-cycle
set $(X,\id_{X})$.

Denote by $\mathsf{HCS}$ the category of Hom-cycle sets, which is
a full subcategory of $\mathsf{HQG}$.

Denote by $\mathsf{S_{hybe}}$ the category of left non-degenerate
involutive solution to HYBE, which is a  full subcategory of $\mathsf{HQS}$.

\begin{thm}\label{thlndihcs}
The functors $G:\mathsf{S_{hybe}}\to\mathsf{HCS}$
and $S:\mathsf{HCS}\to\mathsf{S_{hybe}}$ are mutually inverse,
and so the categories $\mathsf{S_{hybe}}$ and $\mathsf{HCS}$
are isomorphic.
\end{thm}

\begin{proof}
Let $(X,r,\alpha)$ be a left non-degenerate involutive solution to
HYBE. $G(X,r,\alpha)$ is a left Hom-quasigroup.
Furthermore, since $x\cdot y=\lambda_{x}^{-1}(y)$, Theorem~\ref{Sec3ThmLNDIeqV}\eqref{eqlnishybe4}--\eqref{eqlnishybe6}
imply the conditions~\eqref{Sec4ProEq1}--\eqref{Sec4ProEq3}.
Thus $G(X,r,\alpha)$ is a Hom-cycle set.

Conversely, let $(X,\alpha)$ be a Hom-cycle set and $S(X,\alpha)=(X,r,\alpha)$.
Then  $(X,r,\alpha)$ is a left non-degenerate
involutive Hom-quadratic set. By Lemma~\ref{sec2:mark1}\eqref{lndequiv}
and Corollary~\ref{corHQS}, $(X,r,\alpha)$ satisfies
Theorem~\ref{Sec3ThmLNDIeqV}\eqref{eqlnishybe1}--\eqref{eqlnishybe2}. Furthermore, the conditions~\eqref{Sec4ProEq1}--\eqref{Sec4ProEq3}
imply  Theorem~\ref{Sec3ThmLNDIeqV}\eqref{eqlnishybe4}--\eqref{eqlnishybe6}.
Thus $S(X,\alpha)$ is a left non-degenerate involutive solution to
HYBE.
\end{proof}

Theorem~\ref{thlndihcs} generalizes  \cite[Proposition 1]{Rump1}.

\begin{thm}\label{thhcs2}
A left Hom-quasigroup $(X,\alpha)$ is a Hom-cycle
set if and only if the following conditions hold  for all $x,y\in X$,
\begin{gather}
\alpha^{2}(x)\alpha(y)=x\alpha(y),\label{eqthcs2eq1-1}\\
(x\alpha(y))(\alpha(x)\alpha(z))=(y\alpha(x))(\alpha(y)\alpha(z)).\label{eqthcs2eq2}
\end{gather}
\end{thm}

\begin{proof}
Suppose $S(X,\alpha)=(X,r,\alpha)$.

($\Rightarrow$) By Theorem~\ref{thlndihcs}, $(X,r,\alpha)$
is a left non-degenerate involutive solution to HYBE.
Thus~\eqref{eqthcs2eq1-1} follows from Theorem~\ref{Sec3Thm2}(4), and~\eqref{eqthcs2eq2}
follows from~\eqref{Sec4ProEq2}.

($\Leftarrow$)  By Theorem~\ref{thlndihcs}
it suffices to prove that $(X,r,\alpha)$ is a left non-degenerate
involutive solution to HYBE. We need to verify that $(X,r,\alpha)$
satisfies (1)--(5) of Theorem~\ref{Sec3Thm2}. In fact, Lemma~\ref{sec2:mark1}\eqref{lndequiv}
and Corollary~\ref{corHQS} imply that $(X,r,\alpha)$ satisfies
(1)--(3) in Theorem~\ref{Sec3Thm2}, and the conditions~\eqref{eqthcs2eq1-1}
and~\eqref{eqthcs2eq2} imply that $(X,r,\alpha)$ satisfies (4)--(5)
in Theorem~\ref{Sec3Thm2}, as desired.
\end{proof}
\begin{ex} Let $X$ be a nonempty set with a map $\alpha:X\to X$,
and let $f:X\to X$ be a bijection such that $f\alpha=\alpha f$. Define
an operation on $X$ by $xy=f(y)$ for all $x,y\in X$. Then
$(X,\alpha)$ is a Hom-cycle set, which corresponds to the left non-degenerate involutive solution to HYBE defined in Example~\ref{exPermutation}.
 \end{ex}

\begin{ex} Let $(X,\alpha)$ be a left Hom-quasigroup with  $\alpha(X)=\{\theta\}$.
Then $(X,\alpha)$ is a Hom-cycle set if and only if $\theta$ is
a right zero, i.e., $x\theta=\theta$ for all $x\in X$. In this case, $(X,\alpha)$ corresponds to the
solution defined in Example~\ref{exTheta}.\end{ex}

\begin{ex}
A trivial solution to HYBE corresponds to a right zero
groupoid with an endomorphism.
\end{ex}

\begin{ex}\label{exlinearhcs}
   Let $(X,+)$ be an Abelian group with endomorphisms $\varphi,\psi,\alpha$ and define $x\cdot y=\varphi(x)+\psi(y)$ for all  $x,y\in X$. Then $(X,\cdot,\alpha) $ is a Hom-cycle set if and only if the following hold:
   \begin{enumerate}
   \item $\psi$ is bijective;
\item $\varphi\alpha=\alpha\varphi$ and $\psi\alpha=\alpha\psi$;
\item $\varphi\alpha^2=\varphi$;
\item $\varphi^2=\varphi\psi\alpha-\psi\varphi\alpha$.
\end{enumerate}
In fact, (1) is equivalent to saying that $(X,\cdot) $ is a left quasigroup by \cite[Section 4.1]{Bonatto21}; (2) is equivalent to the assertion that $\alpha$ is an endomorphism of $(X,\cdot) $;    (3) and (4) are equivalent to \eqref{eqthcs2eq1-1} and \eqref{eqthcs2eq2}, respectively.
     \end{ex}
\begin{lem}
If $(X,\alpha)$ is a left Hom-quasigroup satisfying~\eqref{eqthcs2eq1-1}, then
\begin{equation}
(x\alpha^{2}(y))\alpha(z)=(xy)\alpha(z),\label{eqderives}
\end{equation}
for all $x,y,z\in X$.
\end{lem}

\begin{proof}
By using~\eqref{eqthcs2eq1-1} we have
\[
(xy)\alpha(z)=\alpha^{2}(xy)\alpha(z)=(\alpha^{2}(x)\alpha^{2}(y))\alpha(z)=(x\alpha^{2}(y))\alpha(z),
\]
for all $x,y,z\in X$.
\end{proof}

\begin{lem}\label{Lemevquival}
If $(X,\alpha)$ is a left Hom-quasigroup   satisfying~\eqref{eqthcs2eq1-1},  then  \eqref{Sec4ProEq1}, \eqref{Sec4ProEq2} and \eqref{Sec4ProEq3} are equivalent.
\end{lem}

\begin{proof}
Replacing $x$ by $\alpha(x)$ in~\eqref{Sec4ProEq1} gives~\eqref{Sec4ProEq2}.

Replacing $y$ by $\alpha(y)$ in~\eqref{Sec4ProEq2} gives~\eqref{Sec4ProEq3}.

Suppose~\eqref{Sec4ProEq3} holds. Replacing $x$ by $\alpha(x)$
and $y$ by $\alpha(y)$ in~\eqref{Sec4ProEq3} gives $\alpha((\alpha^{2}(x)\alpha^{2}(y))(\alpha(x)z))=(\alpha^{2}(y)\alpha^{2}(x))(\alpha(y)\alpha(z))$.
It follows that
\[
\alpha((xy)(\alpha(x)z))=(y\alpha^{2}(x))(\alpha(y)\alpha(z)),
\]
which together with~\eqref{eqderives} gives~\eqref{Sec4ProEq1}.
\end{proof}

\begin{thm}
Let $(X,\alpha)$ be a left Hom-quasigroup. Then
$(X,\alpha)$ is a Hom-cycle set if and only if it satisfies~\eqref{eqthcs2eq1-1}
and one of~\eqref{Sec4ProEq1},~\eqref{Sec4ProEq2} and~\eqref{Sec4ProEq3}.
\end{thm}

\begin{proof}
It follows from Lemma~\ref{Lemevquival} and Theorem~\ref{thhcs2}.
\end{proof}

\begin{thm}\label{thdual}
A Hom-cycle set is non-degenerate if and only if the
corresponding solution to HYBE is non-degenerate.
\end{thm}

\begin{proof}
Let $(X,\alpha)$ be a Hom-cycle set and $(X,r,\alpha)$ the corresponding
solution to HYBE.

($\Rightarrow$) Suppose $(X,\alpha)$ is non-degenerate and $\circ$
is the dual operation defined in Lemma~\ref{lemPhi}. Then~\eqref{eqxyyxc}
can be rewritten as $\lambda_{x\circ y}^{-1}(y\circ x)=x$ for all
$x,y\in X$. Thus $y\circ x=\lambda_{x\circ y}(x)$ for all $x,y\in X$.
It follows by interchanging $x$ and $y$ that $x\circ y=\lambda_{y\circ x}(y)$
for all $x,y\in X$. Substituting the last equation into~\eqref{eqxyyxc},
we get
\[
\lambda_{y\circ x}(y)\cdot\newline(y\circ x)=x.
\]
Denote by $\tau_{y}$ the left multiplication by $y$ with respect
to the operation $\circ$. Then $\rho_{y}\tau_{y}(x)=\rho_{y}(y\circ x)=\lambda_{y\circ x}(y)\cdot(y\circ x)=x$.
Noting that $x\lambda_{x}(y)=y$, by~\eqref{eqxyyxd} we have
\[
\tau_{y}\rho_{y}(x)=(x\lambda_{x}(y))\circ(\lambda_{x}(y)x)=x.
\]
Thus $\rho_{y}$ is bijective, and so $(X,r,\alpha)$ is non-degenerate.

($\Leftarrow$) Suppose $(X,r,\alpha)$ is non-degenerate. Define $x\circ y=\rho_{x}^{-1}(y)$
for all $x,y\in X$. Replacing $y$ by $\lambda_{x}^{-1}(y)$ in~\eqref{sec2:eq1}
we obtain $\rho_{\lambda_{x}^{-1}(y)}(x)=\lambda_{y}^{-1}(x)$, whence
$\rho_{\lambda_{x}^{-1}(y)}^{-1}\lambda_{y}^{-1}(x)=x$, which gives~\eqref{eqxyyxd}. Similarly, replacing $y$ by $\rho_{x}^{-1}(y)$
in~\eqref{sec2:eq2} we can get~\eqref{eqxyyxc}. Thus $(X,\alpha)$
is non-degenerate by Lemma~\ref{lemPhi}.
\end{proof}
Denote by $\mathsf{ndHCS}$ the category of non-degenerate left Hom-cycle
sets. Then $\mathsf{ndHCS}$ is a full subcategory of $\mathsf{HCS}$.

Denote by $\mathsf{ndiS_{hybe}}$ the category of non-degenerate involutive
solutions to HYBE. Then $\mathsf{ndiS_{hybe}}$ is a full subcategory
of $\mathsf{S_{hybe}}$.

Theorem~\ref{thlndihcs} and Theorem~\ref{thdual} have the following
corollary, which generalize \cite[Proposition 2]{Rump1}.
\begin{cor}\label{CorNhcstoHybe}
The functors $G:\mathsf{ndiS_{hybe}}\to\mathsf{ndHCS}$ and $S:\mathsf{ndHCS}\to\mathsf{ndiS_{hybe}}$
are mutually inverse, and so the categories $\mathsf{ndiS_{hybe}}$
and $\mathsf{ndHCS}$ are isomorphic.
\end{cor}

\begin{thm}\label{ProDualcy}
Let $(X,\alpha)$ be a non-degenerate Hom-cycle
set with the dual operation $\circ$. Then $(X,\circ,\alpha)$ is
a non-degenerate Hom-cycle set.
\end{thm}

\begin{proof}
Suppose $S(X,\alpha)=(X,r,\alpha)$. By Theorem~\ref{thdual}, $(X,r,\alpha)$ is a non-degenerate involutive solution to HYBE. Replacing $y$ by $\lambda_{x}(y)$
in~\eqref{eqxyyxd}, we have $(x\lambda_{x}(y))\circ(\lambda_{x}(y)x)=x$,
whence $y\circ\rho_{y}(x)=x$. Replacing $x$ by $\rho_{y}^{-1}(x)$
in the last equation gives $y\circ x=\rho_{y}^{-1}(x)$,  and so $\rho_{y}(y\circ x)=x$. Thus $(X,\circ)$ is a left quasigroup, and we have \[\alpha(x)=\alpha\rho_{y}(y\circ x)=\rho_{\alpha(y)}\alpha(y\circ x),\] whence $\alpha(y\circ x)=\rho_{\alpha(y)}^{-1}\alpha(x)=\alpha(y)\circ\alpha(x)$. Hence $(X,\circ,\alpha)$ is a left Hom-quasigroup. Suppose $S(X,\circ,\alpha)=(X,r^{\circ},\alpha)$. By Corollary~\ref{CorNhcstoHybe}, it suffices to prove that $(X,r^{\circ},\alpha)$ is a non-degenerate involutive solution to HYBE. Since $\lambda_{y}^{\circ -1}(x)=y\circ x=\rho_{y}^{-1}(x)$, we have $\lambda_{y}^{\circ}=\rho_{y}$. Replacing
$y$ by $\rho_{x}(y)$ in~\eqref{eqxyyxc}, we have $(x\circ\rho_{x}(y))(\rho_{x}(y)\circ x)=x$.
Thus \[y(\rho_{x}(y)\circ x)=x=y\lambda_{y}(x),\]  and so we have $\lambda_{y}(x)=\rho_{x}(y)\circ x=\lambda_{x}^{\circ}(y)\circ x=\rho_{y}^{\circ}(x)$, whence $\rho_{y}^{\circ}=\lambda_{y}$.
Thus $r^{\circ}(x,y)=(\rho_{x}(y),\lambda_{y}(x))$.
Clearly $r^{\circ}=\tau r\tau$, where $\tau(x,y)=(y,x)$. Since $(X,r,\alpha)$ is a non-degenerate
involutive solution to HYBE,  so is $(X,r^{\circ},\alpha)$, as desired.
\end{proof}

Finite cycle sets and cycle sets with bijective square maps, especially
square-free cycle sets, are non-degenerate~\cite{Rump1}, but Hom-cycle
sets are not the case.

\begin{ex}\label{ex4order}
Let $X=\{1,2,3,4\}$ with a map $\alpha:X\rightarrow X$
such that $\alpha(X)=\{1\}$. Define an operation on $X$ by the following
multiplication table:
\[
\begin{array}{c|cccc}
\cdot & 1 & 2 & 3 & 4\\
\hline 1 & 1 & 2 & 3 & 4\\
2 & 1 & 2 & 3 & 4\\
3 & 1 & 4 & 3 & 2\\
4 & 1 & 3 & 2 & 4
\end{array}
\]
Then $(X,\cdot,\alpha)$ is a square-free Hom-cycle set, but not non-degenerate
since $\Delta(3,4)=(2,2)=\Delta(2,2)$. \end{ex}

\section{Twists}
Let $(X,\alpha)$ be a left Hom-quasigroup. Define
an operation on $X$ by $x\cdot^{\prime}y=\alpha(x)y$ for all $x,y\in X$. Then $(X,\cdot^{\prime})$ is a left quasigroup,
and
\[
\alpha(x\cdot^{\prime}y)=\alpha\left(\alpha(x)y\right)=\alpha^{2}(x)\alpha(y)=\alpha(x)\cdot^{\prime}\alpha(y).
\]
Thus $\alpha$ is an endomorphism of $(X,\cdot^{\prime})$. Hence
$(X,\cdot^{\prime},\alpha)$ is a left Hom-quasigroup. We call $(X,\cdot^{\prime},\alpha)$
the twist of $(X,\alpha)$, denoted by $T(X,\alpha)$.

Using twist we can define a functor $T:\mathsf{HQG}\to \mathsf{HQG}$ in a natural
way.

A left Hom-quasigroup $(X,\alpha)$ is called an im-cycle set if the
following are satisfied:
\begin{gather}
\alpha^{3}(x)\alpha(y)=\alpha(x)\alpha(y),\label{eqCommon1}\\
(\alpha(x)\alpha(y))(\alpha(x)\alpha(z))=(\alpha(y)\alpha(x))(\alpha(y)\alpha(z)),\label{eqCommon2}
\end{gather}
for all $x,y,z\in X$.

It is clear that $(X,\alpha)$ is an im-cycle set if and only if $\alpha(X)$ is a cycle set and~\eqref{eqCommon1} holds.

\begin{thm}
\label{ThmStableCS} A left Hom-quasigroup is an im-cycle set if and only if its twist is a Hom-cycle set.
 \end{thm}

\begin{proof}Let $(X,\cdot,\alpha)$ be a left Hom-quasigroup. Then its twist $(X,\cdot^{\prime},\alpha)$ is also a left Hom-quasigroup.

($\Rightarrow$)  By~\eqref{eqCommon1} we have $\alpha^{2}(x)\cdot^{\prime}\alpha(y)=x\cdot^{\prime}\alpha(y)$
for all $x,y\in X$. Replacing $x$ by $\alpha(x)$ and $y$ by $\alpha(y)$
in~\eqref{eqCommon2}, we obtain
\[
(\alpha^{2}(x)\alpha^{2}(y))(\alpha^{2}(x)\alpha(z)) =(\alpha^{2}(y)\alpha^{2}(x))(\alpha^{2}(y)\alpha(z)),
\]
which implies that $(x\cdot^{\prime}\alpha(y))\cdot^{\prime}(\alpha(x)\cdot^{\prime}\alpha(z)) =(y\cdot^{\prime}\alpha(x))\cdot^{\prime}(\alpha(y)\cdot^{\prime}\alpha(z))$.
By Theorem~\ref{thhcs2}, $(X,\cdot^{\prime},\alpha)$ is a Hom-cycle
set.

 ($\Leftarrow$) Since $(X,\cdot^{\prime},\alpha)$
is a Hom-cycle set, applying Theorem~\ref{thhcs2} to $(X,\cdot^{\prime},\alpha)$ yields~\eqref{eqCommon1} and
 \[
\alpha(\alpha(x)\alpha(y))(\alpha^{2}(x)\alpha(z)) =\alpha(\alpha(y)\alpha(x))(\alpha^{2}(y)\alpha(z)).
\]
Replacing $x$ by $\alpha(x)$ and $y$ by $\alpha(y)$ in the last
equation, we have
\[
(\alpha^{3}(x)\alpha^{3}(y))(\alpha^{3}(x)\alpha(z)) =(\alpha^{3}(y)\alpha^{3}(x))(\alpha^{3}(y)\alpha(z)).
\]
By~\eqref{eqCommon1}, we see that $(\alpha(X),\cdot)$
is a cycle set. Thus $(X,\alpha)$ is an im-cycle set.
\end{proof}

\begin{cor}
  The twist of an im-cycle set is a Hom-cycle set.
\end{cor}

\begin{cor}\label{CorSCtoHSC}
Let $X$ be a cycle set with an endomorphism $\alpha$ such that~\eqref{eqCommon1} holds. Then  the twist of $(X,\alpha)$ is a Hom-cycle set.
\end{cor}

\begin{thm}\label{Thmhcycle}
The twist of a Hom-cycle set is an im-cycle set.
\end{thm}

\begin{proof}
Let $(X,\cdot,\alpha)$ be a Hom-cycle set with the twist $(X,\cdot',\alpha)$.
By~\eqref{eqthcs2eq1-1}, we can rewrite~\eqref{eqthcs2eq2} as
\begin{equation*}\alpha^{2}(x\alpha(y))(\alpha(x)\alpha(z))=\alpha^{2}(y\alpha(x))(\alpha(y)\alpha(z)),\end{equation*}
that is,
\[
(\alpha^{2}(x)\alpha^{3}(y))(\alpha(x)\alpha(z))=(\alpha^{2}(y)\alpha^{3}(x))(\alpha(y)\alpha(z)),
\]
for all $x,y,z\in X$. Thus by~\eqref{eqderives}, we have
\[
(\alpha^{2}(x)\alpha(y))(\alpha(x)\alpha(z))=(\alpha^{2}(y)\alpha(x))(\alpha(y)\alpha(z)),
\]
which implies $(x\cdot^{\prime}{y})\cdot^{\prime}(x\cdot^{\prime}\alpha(z))=(y\cdot^{\prime}x)\cdot^{\prime}(y\cdot^{\prime}\alpha(z))$
for all $x,y,z\in X$. Hence $(\alpha(X),\cdot^{\prime})$ is a cycle
set. By~\eqref{eqthcs2eq1-1}, we have
\[\alpha^{4}(x)\cdot \alpha(y)=\alpha^{2}(x)\cdot \alpha(y)\]
 for all $x,y\in X$, whence $\alpha^{3}(x)\cdot' \alpha(y)=\alpha(x)\cdot' \alpha(y)$ for all $x,y\in X$. Thus $(X,\cdot',\alpha)$ is an im-cycle set.
\end{proof}

\begin{thm}\label{thndtw}
If $(X,\alpha)$ be a non-degenerate Hom-cycle set,
then its twist is non-degenerate.
\end{thm}

\begin{proof}
Let $T(X,\alpha)=(X,\cdot^{\prime},\alpha)$.
By Theorem~\ref{ProDualcy}, $(X,\circ,\alpha)$ is a Hom-cycle set.
Applying Theorem~\ref{thhcs2} to $(X,\circ,\alpha)$ we obtain
\begin{equation}
\alpha^{2}(x)\circ\alpha(y)=x\circ\alpha(y),\label{eqcirc2}
\end{equation}
for all $x,y\in X$. Define $x\circ^{\prime}y=\alpha(x)\circ y$ for
all $x,y\in X$. Then by~\eqref{eqthcs2eq1-1},~\eqref{eqcirc2} and
Lemma~\ref{lemPhi} we have
\begin{gather*}
(x\cdot^{\prime}y)\circ^{\prime}(y\cdot^{\prime}x)=\alpha(\alpha(x)y)\circ(\alpha(y)x)=(x\alpha(y))\circ(\alpha(y)x)=x,\\
(x\circ^{\prime}y)\cdot^{\prime}(y\circ^{\prime}x)=\alpha(\alpha(x)\circ y)(\alpha(y)\circ x)=(x\circ\alpha(y))(\alpha(y)\circ x)=x.
\end{gather*}
Thus $(X,\cdot^{\prime},\alpha)$ is non-degenerate by Lemma~\ref{lemPhi}.
\end{proof}

\begin{lem}\label{lemsingnon}
Let $(X,\alpha)$ be a left Hom-quasigroup with
$\alpha(X)$ singleton. Then its twist is a non-degenerate left Hom-quasigroup.
\end{lem}

\begin{proof}
Let $\alpha(X)=\{\theta\}$. Then $\theta$ is a right zero of the
left quasigroup $X$. Let $T(X,\alpha)=(X,\cdot^{\prime},\alpha)$.
Thus
\[\Delta(x,y)=(x\cdot^{\prime}y,y\cdot^{\prime}x)=(\alpha(x)y, \alpha(y)x)=(\theta y,\theta x )\]
for all $x,y \in X$. It follows $\Delta=\sigma_{\theta}\times \sigma_{\theta}$. Since $\sigma_{\theta}$  is bijective, so is $\Delta$. Hence $(X,\cdot^{\prime},\alpha)$ is non-degenerate.
\end{proof}
\begin{rem}
A Hom-cycle set is not necessarily the twist of an im-cycle set, and
vice versa. For example, let $(X,\alpha)$ be as in Example~\ref{ex4order}.
It is degenerate and both a Hom-cycle set and an im-cycle set, but
it is isomorphic to neither the twist of a Hom-cycle set nor the twist
of an im-cycle set by Lemma~\ref{lemsingnon}.
\end{rem}

The following example shows that twist of a Hom-cycle set is not necessarily a Hom-cycle set.
\begin{ex}\label{examMatrix}
 Let $F$ be a field of characteristic $ \neq 2$, $X$ the vector space ${F}^{3} $ and $\varphi,\psi, \alpha$   the linear endomorphisms of $X$ defined under the natural basis by the following matrices, respectively,    \[A=\begin{pmatrix}
      0 & 1 & 0 \\
      0 & 0 & 1 \\
      0 & 0 & 0
    \end{pmatrix},~B=\begin{pmatrix}
      1 & 0 & 0 \\
0 & 1 & 1 \\
0 & 0 & 1
    \end{pmatrix}, ~C=\begin{pmatrix}
      -1 & 0 & 1 \\
0 & -1 & 0 \\
0 & 0 & -1
    \end{pmatrix}.\]
  Then it is easy to check that $\psi$  is bijective and \[\
 \varphi\alpha=\alpha\varphi , ~\psi\alpha=\alpha\psi,~
 \varphi\alpha^2=\varphi,~\varphi^2=\varphi\psi -\psi\varphi,~\varphi^2\alpha\ne\varphi^2.   \]    Define $x\cdot y=\varphi(x)+\psi(y)$ for all $x,y\in X$.  Then   $(X,\cdot)$ is a cycle set by \cite[Section 4.1]{Bonatto21} and $\alpha$ is an endomorphism of $(X,\cdot)$ satisfying $\alpha^2(x)y=xy$ for all $x,y\in X$.

 The twist $T (X,\cdot,\alpha)$ of $(X,\cdot,\alpha )$ is a Hom-cycle set By Corollary \ref{CorSCtoHSC}.  Clearly, the twist of $T (X,\cdot,\alpha)$ equals  $(X,\cdot,\alpha )$, which is not a Hom-cycle set since (4) in Example \ref{exlinearhcs} does not hold.
\end{ex}

\begin{cor}\label{corq'}
If $(X,\alpha)$ is a non-degenerate Hom-cycle set,
then the map $q^{\prime}:X\to X,~~x\mapsto\alpha(x)x$ is bijective.
\end{cor}

\begin{proof}
By Theorem~\ref{thndtw} the twist $(X,\cdot',\alpha)$ is non-degenerate. Since $q^{\prime}$ is the square map of $(X,\cdot')$, $q^{\prime}$ is  bijective by Lemma~\ref{lem2div}.
\end{proof}

\begin{rem}
The converses of Theorem~\ref{thndtw} and Corollary~\ref{corq'}
do not hold. Example~\ref{ex4order} provides  a counterexample to the
both cases.
\end{rem}

We now apply Theorem~\ref{ThmStableCS}, Corollary~\ref{CorSCtoHSC} and Theorem~\ref{Thmhcycle} to  solutions to HYBE.

Let $(X,r,\alpha)$ be a left non-degenerate involutive Hom-quadratic
set. We call  $STG(X,r,\alpha)$ the twist of $(X,r,\alpha)$.  The twist of $(X,r,\alpha)$ is also a left non-degenerate involutive Hom-quadratic
set.
\begin{thm}
Let $(X,r,\alpha)$ be a left non-degenerate involutive Hom-quadratic
set with the twist $(X,r',\alpha)$. Then
\begin{equation}\label{eqrprime}
r^{\prime}(x,y)=(\lambda_{\alpha(x)}(y),
\lambda_{\lambda_{\alpha^{2}(x)}\alpha(y)}^{-1}(x))
\end{equation}
 for all $x,y\in X$. If $(X,r,\alpha)$ is additionally a solution to HYBE, then
 \begin{equation}\label{eqrprime2}
r^{\prime}(x,y)=(\lambda_{\alpha(x)}(y),
\rho_{ \alpha(y)} (x))
\end{equation}for all $x,y\in X$.
\end{thm}

\begin{proof} Since $ STG(X,r,\alpha)=(X,r',\alpha)$, by Theorem~\ref{LemlnditP}   we have  \[  TG(X,r,\alpha)=G(X,r',\alpha).\]
 Thus $G(X,r',\alpha)$ is the twist of $   G(X,r,\alpha)$. Let $G(X,r,\alpha)=(X,\cdot,\alpha)$.  Then  $G(X,r',\alpha)=(X,\cdot',\alpha)$,   where $x\cdot'y=\alpha(x)y$ for all $x,y\in X$. Thus $\lambda'_x =\lambda_{\alpha(x)}$, and so
\[\rho'_y(x)=\lambda'^{-1}_{\lambda'_{x}(y)}(x)
=\lambda ^{-1}_{\alpha\lambda _{\alpha(x)}(y)}(x)
=\lambda ^{-1}_{\lambda _{\alpha^2(x)}\alpha(y)}(x). \]
Therefore~\eqref{eqrprime} holds.

Furthermore, if $(X,r,\alpha)$ is  a solution to HYBE, then~\eqref{eqrprime2} follows from~\eqref{eqrprime}, Corollary~\ref{corxa2}\eqref{eqcor1} and~\eqref{eqInvolutive}.
\end{proof}

\begin{rem}
   By Example~\ref{examMatrix} and Theorem~\ref{thlndihcs},  the twist of a  left non-degenerate involutive solution to HYBE is not necessarily a solution to HYBE.
\end{rem}

\begin{lem}\label{ThmLndHqs}
Let $(X,r,\alpha)$ be a left non-degenerate involutive
Hom-quadratic set and $i,j$ be nonnegative integers such that $0\leq j-i\leq2$.
Then the following are equivalent.
\begin{enumerate}
  \item $r(\alpha^{i+2}\times\alpha^{j})
      =(1\times\alpha^{2})r(\alpha^{i}\times\alpha^{j})$;
  \item $\lambda_{\alpha^{i+2}(x)}\alpha^{j}=\lambda_{\alpha^{i}(x)}\alpha^{j}$ for all $x\in X $.
\end{enumerate}
\end{lem}

\begin{proof}
We first note that for all $x,y\in X$,
\begin{gather*}
r(\alpha^{i+2}\times\alpha^{j})(x,y)=(\lambda_{\alpha^{i+2}(x)}\alpha^{j}(y),\rho_{\alpha^{j}(y)}\alpha^{i+2}(x)),\\
(1\times\alpha^{2})r(\alpha^{i}\times\alpha^{j})(x,y)=(\lambda_{\alpha^{i}(x)}\alpha^{j}(y),\alpha^{2}\rho_{\alpha^{j}(y)}\alpha^{i}(x)).
\end{gather*}
It suffices to prove that  (2) implies $\rho_{\alpha^{j}(y)}\alpha^{i+2}(x)=\alpha^{2}\rho_{\alpha^{j}(y)}\alpha^{i}(x)$
for all $x,y\in X$. Since $0\leq j-i\leq2$, we have $\lambda_{\alpha^{i}(x)}\alpha^{j}=\alpha^{j}\lambda_{\alpha^{i+2-j}(x)}$,
and so $\alpha^{j}\lambda_{\alpha^{i+2-j}(x)}^{-1}=\lambda_{\alpha^{i}(x)}^{-1}\alpha^{j}$.
Thus $\lambda_{\alpha^{i+2}(x)}^{-1}\alpha^{j}=\lambda_{\alpha^{i}(x)}^{-1}\alpha^{j}$.
Consequently,
\begin{multline*}
\alpha^{2}\rho_{\alpha^{j}(y)}\alpha^{i}(x)=\lambda_{\lambda_{\alpha^{i+2}(x)}\alpha^{j+2}(y)}^{-1}\alpha^{i+2}(x)=\lambda_{\alpha^{i+2}(\lambda_{x}\alpha^{j-i}(y))}^{-1}\alpha^{i+2}(x)\\
=\lambda_{\alpha^{i}(\lambda_{x}\alpha^{j-i}(y))}^{-1}\alpha^{i+2}(x)=\lambda_{\lambda_{\alpha^{i}(x)}\alpha^{j}(y)}^{-1}\alpha^{i+2}(x)\\
=\lambda_{\lambda_{\alpha^{i+2}(x)}\alpha^{j}(y)}^{-1}\alpha^{i+2}(x)=\rho_{\alpha^{j}(y)}\alpha^{i+2}(x),
\end{multline*}
as desired.
\end{proof}
\begin{cor}
Let $(X,r,\alpha)$ be a  left non-degenerate involutive solution to HYBE.
 Then $r(\alpha^{2}\times\alpha)=(\id\times\alpha^{2})r(\id\times\alpha)$.
 \end{cor}
 \begin{proof}
    It follows from Corollary~\ref{corxa2} and Lemma~\ref{ThmLndHqs}.
 \end{proof}
\begin{thm}
Let $(X,r,\alpha)$ be left non-degenerate involutive Hom-quadratic set with the twist
$(X,r^{\prime},\alpha)$.
\begin{enumerate}
\item If $(X,r,\alpha)$ is a  solution to
HYBE, then $(\alpha(X),r^{\prime})$ is a solution to YBE and $r'(\alpha^{2}\times\id)=(\id\times\alpha^{2})r'$
on $\alpha(X)$.
\item The twist $(X,r^{\prime},\alpha)$ is a solution to HYBE
if  and only if  $(\alpha(X),r)$ is a solution to YBE
and  $r(\alpha^{2}\times\id)=(\id\times\alpha^{2})r$ on $\alpha(X)$.
\item If $(X,r)$ is a  solution to YBE and $r(\alpha^{3}\times\alpha)=(\alpha\times\alpha^{3})r$, then
$(X,r^{\prime},\alpha)$ is a  solution to HYBE.
\end{enumerate}
\end{thm}

\begin{proof} Let $S(X,r,\alpha)=(X,\cdot,\alpha)$. Then  $S(X,r^{\prime},\alpha)=(X,\cdot',\alpha)$,   the twist of $(X,\cdot,\alpha)$.

(1)   Theorem~\ref{thlndihcs} implies that $(X,\cdot,\alpha)$ is a Hom-cycle set. By Theorem~\ref{Thmhcycle}, $(\alpha(X),\cdot') $ is a cycle set and~\eqref{eqCommon1} holds with respect to the operation $\cdot'$. It follows that $\alpha^{2}(x)\cdot' y=x\cdot' y$ for all $x,y\in \alpha(X)$, which implies  $r'(\alpha^{2}\times\id)=(\id\times\alpha^{2})r'$ on $\alpha(X)$ by Lemma~\ref{ThmLndHqs}. By Theorem~\ref{TemLndisC}, $(\alpha(X),r') $ is a solution to YBE.

(2) By Theorem~\ref{thlndihcs},  $(X,r^{\prime},\alpha)$ is a  left non-degenerate involutive solution to HYBE  if and only if $(X,\cdot',\alpha)$ is a Hom-cycle set. Equivalently,  by Theorem~\ref{ThmStableCS} $(\alpha(X),\cdot)$ is a left cycle set satisfying~\eqref{eqCommon1}. This is equivalent to that $(\alpha(X),r)$ is a solution to YBE satisfying $r(\alpha^{2}\times\id)=(\id\times\alpha^{2})r$ on $\alpha(X)$  by Theorem~\ref{TemLndisC} and Lemma~\ref{ThmLndHqs}.

(3)  follows from the sufficiency in  (2).
 \end{proof}

\section*{Acknowledgements}

This work is supported by NSF of China (No.12171194, No.11971289).
 \bibliographystyle{abbrvnat}
\bibliography{20221201}

\end{document}